\documentclass[a4paper,11pt]{amsart}

\usepackage{amsmath,amssymb,amsthm,mathtools}
\usepackage{microtype}
\usepackage[colorlinks=true,linkcolor=blue,citecolor=blue,urlcolor=blue]{hyperref}

\allowdisplaybreaks
\numberwithin{equation}{section}
\newtheorem{theorem}{Theorem}[section]
\newtheorem{lemma}[theorem]{Lemma}
\newtheorem{corollary}[theorem]{Corollary}
\newtheorem{remark}[theorem]{Remark}
\theoremstyle{definition}
\newtheorem{conjecture}[theorem]{Conjecture}

\newcommand{\Sph}{\mathbb S}
\newcommand{\R}{\mathbb R}
\newcommand{\tr}{\operatorname{tr}}

\title[Lu's conjecture for minimal surfaces in codimension two]
{Lu's conjecture for minimal surfaces in codimension two}

\author[J. Q. Ge]{Jianquan Ge}
\address{School of Mathematical Sciences, Laboratory of Mathematics and Complex Systems, Beijing Normal University, Beijing 100875, P. R. China}
\email{jqge@bnu.edu.cn}

\author[F. G. Li]{Fagui Li}
\address{Frontier Interdisciplinary Domain, Beijing Institute of Technology, Zhuhai, Guangdong 519088, P. R. China}
\email{lifagui@bitzh.edu.cn}

\author[Y. H. Zhang]{Yunheng Zhang$^{*}$}
\address{$^{*}$School of Mathematical Sciences, Laboratory of Mathematics and Complex Systems, Beijing Normal University, Beijing 100875, P. R. China}
\email{yunheng@mail.bnu.edu.cn}
\date{}
\thanks{$^{*}$ Corresponding author.}
\thanks{J. Q. Ge is partially supported by NSFC (No. 12571049) and the Fundamental Research Funds for the Central Universities.}
\thanks{F. G. Li is partially supported by NSFC (Nos. 12271040 and 12501061), the Guangdong Provincial Association for Science and Technology Youth Talent Support Program (No. SKXRC2026413) and the Research Start-up Funding of Beijing Institute of Technology (No. 5640011253301).}
\begin{document}

\begin{abstract}
Let $M^2\to\Sph^4$ be a closed minimal immersion, let $S$ be the
squared norm of its second fundamental form, and let
$\lambda_1\geq\lambda_2\geq0$ be the eigenvalues of Lu's fundamental matrix.
We classify all such immersions for which $S+\lambda_2$ is constant.  We
prove that the constant can only be $0$ or $2$.  In the first case the image
is a totally geodesic $2$-sphere; in the second case it is either a Clifford
torus in a totally geodesic $\Sph^3$ or the Veronese surface in $\Sph^4$.
In particular, there is no closed minimal surface in $\Sph^4$ with constant
$S+\lambda_2>2$.  Consequently, Lu's second-gap conjecture holds for
minimal surfaces in codimension two.  Together with the hypersurface result
of Peng--Terng and the counterexamples of Li--Zhao in every codimension
$m\geq3$, this completes the codimension picture for minimal surfaces.
\end{abstract}

\maketitle

\section{Introduction}

Let $M^n$ be a closed minimal submanifold of the unit sphere
$\Sph^{n+m}$. 
Unless explicitly stated otherwise, all submanifolds in this paper are connected and
have no boundary. 
 With respect to a local orthonormal normal frame
$\{e_\alpha\}_{\alpha=n+1}^{n+m}$, let $A^\alpha$ denote the corresponding
shape operators.  Lu \cite{Lu11} introduced the fundamental matrix
\[
        \mathcal A
        =\bigl(\langle A^\alpha,A^\beta\rangle\bigr)_{m\times m}.
\]
If
$\lambda_1\geq\lambda_2\geq\cdots\geq\lambda_m\geq0$
are the eigenvalues of $\mathcal A$, then
\[
        S=\tr\mathcal A
          =\lambda_1+\cdots+\lambda_m.
\]  
The renowned DDVV conjecture, proposed by De Smet, Dillen,
Verstraelen, and Vrancken \cite{DDVV}, asserts that for an immersed submanifold of a real space form with constant sectional curvature $c$, the following pointwise inequality holds: 
\[\rho+\rho^{\perp}\leq|H|^2+c,\]
where $\rho$ denotes the normalized scalar curvature, $\rho^{\perp}$ denotes the normalized normal scalar curvature, and $H$ denotes the normalized mean curvature vector field.
By the Gauss equation
and the Ricci equation, this is equivalent to the following algebraic inequality:		 
		\begin{equation*}
			\sum_{r,s=1}^{m}\|[B_{r},B_{s}]\|^2\leq(\sum_{r=1}^{m}\|B_{r}\|^2)^2,
		\end{equation*} 
	where $B_{1},\ldots,B_{m}$ are $n\times n$ real symmetric matrices, $\|\cdot\|^2$ denotes the Frobenius norm and $[B_{r},B_{s}]:=B_{r}B_{s}-B_{s}B_{r}$ is the commutator of the matrices $B_{r}$ and $B_{s}$. The DDVV inequality was independently proved by Ge and Tang \cite{GT} and by Lu \cite{Lu11}; see also the survey \cite{GLTZ} and the related work \cite{GLZ} for more details. Later, Lu used a generalized DDVV inequality for commutators of the shape operators and derived the following first-gap theorem involving the refined curvature quantity $S+\lambda_2$, which interpolates between the hypersurface and higher-codimensional pinching problems.
\begin{theorem}[Lu's first-gap theorem \cite{Lu11}]
Let $M^n$ be a closed minimal submanifold of $\Sph^{n+m}$.  If
\[
        0\leq S+\lambda_2\leq n,
\]
then $M$ is totally geodesic, a Clifford torus in a totally geodesic
$\Sph^{n+1}$, or the Veronese surface in a totally geodesic $\Sph^4$.
\end{theorem}
The above pinching theorem generalizes earlier pinching theorems of Simons \cite{Sim}, Chern, do Carmo, and Kobayashi \cite{CDK}, and Li and Li \cite{1992 LiLi}. Motivated by a classical gap theorem established by Peng and Terng \cite{Peng-T1},
Lu also proposed the following second-gap conjecture.
\begin{conjecture}[Lu's second-gap conjecture \cite{Lu11}]\label{conj:lu-second-gap}
For every pair $(n,m)$ there exists $\varepsilon(n,m)>0$ with the following
property.  If $M^n$ is a closed minimal submanifold of $\Sph^{n+m}$ and
$S+\lambda_2$ is constant with $S+\lambda_2>n$, then
\[
        S+\lambda_2>n+\varepsilon(n,m).
\]
\end{conjecture}
Indeed, the case $m=1$ was settled by Peng and Terng \cite{Peng-T1} and is closely related to the famous Chern conjecture: \emph{for a closed immersed minimal hypersurface $M^n$ in the unit sphere $\mathbb{S}^{n+1}$ with constant scalar curvature $R_M$, the set of all possible values for $R_M$ is discrete for each $n$}. In 1993, Chang \cite{C2} proved that \emph{a closed minimally immersed hypersurface with constant scalar curvature in $\mathbb{S}^{4}$ is isoparametric}, thereby resolving the stronger version of the Chern conjecture for $n=3$. Tang--Wei--Yan \cite{T-W-Y} and Tang--Yan \cite{T-Y} established strong sufficient conditions for hypersurfaces with prescribed higher mean curvatures or trace invariants to be isoparametric. More recently, under additional assumptions, Ge--Tan--Yan--Zhang \cite{GeTanYanZhang25}, He--Xu--Zhao \cite{HeXuZhao26}, Tao \cite{Tao26}, Ge--Liu--Luo--Yan \cite{GLLY} and Deng--Kou \cite{DengKou2026} have also made progress on the Chern conjecture. In higher codimension, Ge--Li--Zhang \cite{GeLiZhang2026} obtained an explicit second-gap rigidity theorem under the flat-normal-bundle assumption.

Very recently, Ding--Ge--Li--Yang \cite{DingGeLiYang2026} 
proved Lu's conjecture for minimal $2$-spheres and for general surfaces under mild assumptions on the normal scalar curvature. In particular, they obtained the following result.
\begin{theorem}[\cite{DingGeLiYang2026}]
	Let $M^2$ be a closed minimal sphere immersed in $\mathbb{S}^{2+m}$.
	\begin{enumerate}
		\item[(1)] If $S + \lambda_{2}$ is constant, then $S+\lambda_{2}=\frac{3(s-1)(s+2)}{s(s+1)}$, where $s$ is a positive integer, the curvature $K$ is constant and the submanifold is one of
		Calabi's minimal $2$-spheres \cite{Calabi1967,DoCarmoWallach1971}.
		\item[(2)] If $S+\lambda_2>2$, then
\[
        \max_{p\in M}(S+\lambda_2)(p)\geq\frac{5}{2}.
\]
	\end{enumerate}
\end{theorem}

%For minimal two-spheres, Ding--Ge--Li--Yang \cite{DingGeLiYang2026}
%proved that constancy of $S+\lambda_2$ forces the Gauss curvature to be
%constant and hence yields one of Calabi's minimal two-spheres.  They also
%proved that, if $S+\lambda_2>2$ on a minimal two-sphere, then
%\[
%        \max_M(S+\lambda_2)\geq\frac52.
%\]
On the other hand, Li--Zhao \cite{LiZhao2026} constructed closed
embedded counterexamples in every codimension $m\geq3$.  More precisely, in
every odd codimension $m\geq3$, they obtained linearly full embedded flat
minimal tori whose constant values of $S+\lambda_2$ are dense in the interval $(2,3)$.
For every even codimension $m\geq4$, counterexamples follow by composing the
codimension-three examples with a totally geodesic inclusion into a higher
dimensional sphere; the added normal directions have zero shape operators,
so $S$ and $\lambda_2$ are unchanged.  These even-codimensional examples are
not linearly full in the larger ambient sphere.  Thus, among minimal surfaces,
the only codimension not excluded by the known counterexamples is codimension
two, namely minimal surfaces in $\mathbb{S}^4$.

Our main result gives a complete classification in this remaining
codimension.

\begin{theorem}\label{thm:main}
Let $M^2$ be a closed minimal surface immersed in $\mathbb{S}^{4}$ and suppose that
\[
        S+\lambda_2\equiv c
\]
for a constant $c$.  Then $c\in\{0,2\}$.  More precisely, exactly one of the following occurs:
\begin{enumerate}
\item[(i)] $c=0$ and $M$ is a totally geodesic $2$-sphere;
\item[(ii)] $c=2$ and $M$ is a Clifford torus in a totally geodesic
$\Sph^3\subset\Sph^4$;
\item[(iii)] $c=2$ and $M$ is the Veronese surface in $\Sph^4$.
\end{enumerate}
\end{theorem}

\begin{corollary}\label{cor:lu-22}
There is no closed minimal immersion $M^2\to\Sph^4$ for which
$S+\lambda_2$ is a constant larger than $2$.  Hence Lu's second-gap
conjecture holds for $(n,m)=(2,2)$.  Moreover, for minimal surfaces, the
conjecture holds precisely in codimensions $m=1,2$ and fails in every
codimension $m\geq3$.
\end{corollary}

\begin{remark}
The codimension-two conclusion does not extend to all even codimensions.  For
$q\geq2$, the linearly full Calabi minimal two-sphere \cite{Calabi1967,DoCarmoWallach1971} in
$\Sph^{2q}$ has
\[
        K=\frac{2}{q(q+1)},
        \qquad
        S=2(1-K),
        \qquad
        \lambda_1=\lambda_2=\frac S2.
\]
Consequently,
\[
        S+\lambda_2
        =3\left(1-\frac{2}{q(q+1)}\right).
\]
This value equals $2$ when $q=2$ and belongs to $(2,3)$ for every $q\geq3$.
\end{remark}

The paper is organized as follows.  In Section~2, we recall the basic notation
and formulas for minimal surfaces in $\Sph^4$.  In Section~3, we derive the
local differential identity needed in the proof.  In Section~4, we prove the classification theorem.

\section{Preliminaries}\label{pre}

Let $F:M^2\to\Sph^4$ be a minimal immersion.  All frame computations are
local.  When $M$ is nonorientable, they may equivalently be carried out on its
oriented double cover.  Choose an adapted orthonormal frame
\[
        e_1,e_2\in TM,
        \qquad
        e_3,e_4\in NM,
\]
with dual coframe $\omega_1,\omega_2$.  We use the conventions
\[
        *\omega_1=\omega_2,
        \qquad
        *\omega_2=-\omega_1,
        \qquad
        d\mu=\omega_1\wedge\omega_2.
\]
For a smooth function $u$, write
\[
        du=u_1\omega_1+u_2\omega_2,
        \qquad u_i=e_i(u).
\]
Then
\[
        |\nabla u|^2=u_1^2+u_2^2,
        \qquad
        *du=u_1\omega_2-u_2\omega_1.
\]
We use the convention
$\Delta u=\operatorname{div}(\nabla u)$; in particular,
$\Delta u\leq0$ at a local maximum. For every smooth function $f=f(b)$, a direct computation shows that
\begin{equation}\label{eq:star-divergence}
        d\bigl(f(b)*db\bigr)
        =\operatorname{div}\bigl(f(b)\nabla b\bigr)d\mu
        =\bigl(f(b)\Delta b+f'(b)|\nabla b|^2\bigr)d\mu.
\end{equation}

Let the range of the indices be as follows:
\begin{equation*}
	1\leq i,j,k,\ldots\leq 2,\qquad 3\leq\alpha,\beta,\gamma,\ldots\leq 4,\qquad 1\leq A,B,C,\ldots\leq4.
\end{equation*} 
We next fix the signs of the curvature forms used below. The connection forms
are defined by
\[
        de_A=\sum_B\omega_{AB}e_B,
        \qquad
        \omega_{AB}+\omega_{BA}=0.
\]
After restriction to $M$, Cartan's second structural equations in the unit
sphere read
\[
        d\omega_{AB}
        =\sum_C\omega_{AC}\wedge\omega_{CB}-\omega_A\wedge\omega_B,
        \qquad \omega_3=\omega_4=0.
\]
Writing
\[
        \omega_{i\alpha}
        =\sum_{j=1}^2h^\alpha_{ij}\omega_j,
        \qquad h^\alpha_{ij}=h^\alpha_{ji},
\]
we obtain
\begin{align*}
        d\omega_{12}
        &=-\omega_1\wedge\omega_2
          -\sum_{\alpha=3}^4\omega_{1\alpha}\wedge\omega_{2\alpha} \\
        &=-\left\{1+\sum_{\alpha=3}^4
          \bigl(h^\alpha_{11}h^\alpha_{22}-(h^\alpha_{12})^2\bigr)
          \right\}d\mu.
\end{align*}
The Gauss equation therefore gives
\begin{equation}\label{eq:gauss-curvature-form}
       K=1+\sum_{\alpha=3}^4
       \bigl(h^\alpha_{11}h^\alpha_{22}-(h^\alpha_{12})^2\bigr).
\end{equation}
Similarly,
\begin{align*}
        d\omega_{34}
        &=-\omega_{13}\wedge\omega_{14}
          -\omega_{23}\wedge\omega_{24} \\
        &=-\Bigl(
          h^3_{11}h^4_{12}-h^3_{12}h^4_{11}
          +h^3_{12}h^4_{22}-h^3_{22}h^4_{12}
          \Bigr)d\mu.
\end{align*}
Accordingly, we define the signed normal curvature by
\begin{equation}\label{eq:normal-curvature-form}
        d\omega_{34}=-K^\perp d\mu,
\end{equation}
which yields the Ricci equation:
\begin{equation}\label{Ricci}
        K^\perp
        =h^3_{11}h^4_{12}-h^3_{12}h^4_{11}
         +h^3_{12}h^4_{22}-h^3_{22}h^4_{12}.
\end{equation}
Denote by $A^{\alpha}$ the shape operator of $M^2$ with respect to a given normal orthonormal frame. For later use, we assume that $\lambda_{1}\geqslant \lambda_{2}$ are eigenvalues of the fundamental matrix $\mathcal{A}=(\langle A^{\alpha},A^{\beta}\rangle)_{2\times 2}$.
Consider the open set
\[
        U:=\{p\in M:\lambda_1(p)>\lambda_2(p)>0\}.
\]
On a sufficiently small neighborhood of every point $p$ in $U$, we may choose tangent and normal
frames such that
\begin{equation}\label{eq:canonical-shapes}
        A^3=\begin{pmatrix}a&0\\0&-a\end{pmatrix},
        \qquad
        A^4=\begin{pmatrix}0&b\\b&0\end{pmatrix},
        \qquad a>b>0.
\end{equation}
Then
\[
        \lambda_1=2a^2,
        \qquad
        \lambda_2=2b^2,
        \qquad
        S=2a^2+2b^2.
\]
Thus $a$ and $b$ are globally well‑defined on $U$, and because the fundamental matrix $\mathcal{A}$ has two simple eigenvalues, which are smooth, so $a$ and $b$ are smooth as well.  The Gauss and Ricci equations \eqref{eq:gauss-curvature-form} and \eqref{Ricci} give
\begin{equation}\label{eq:K-and-Kperp}
        K=1-a^2-b^2,
        \qquad
        K^\perp=2ab.
\end{equation}
Here the tangent and normal orientations have been chosen, on each local
neighborhood, so that the second identity has the displayed sign.  Only this
local convention is used in the sequel.

Assume from now on that
\[
        S+\lambda_2\equiv c,
        \qquad c>2.
\]
On $U$, this is equivalent to
\begin{equation}\label{eq:a-b-constraint}
        2a^2+4b^2=c.%,
%        \qquad\text{or}\qquad
    %    a^2+2b^2=\frac c2.
\end{equation}
Set
\[
        D:=a^2-b^2=\frac c2-3b^2.
\]
Since $a>b>0$,
\begin{equation}\label{eq:b-range}
        0<b^2<\frac c6.
\end{equation}

\section{The local differential identity}

The following computation is the core of the proof.

\begin{lemma}\label{lem:connections}
On $U$ one has
\begin{equation}\label{eq:connection-formulas}
        \omega_{12}=\frac{3b}{2D}*db,
        \qquad
        \omega_{34}=\frac{c}{2aD}*db.
\end{equation}
\end{lemma}

\begin{proof}
Write
\[
        \omega_{12}=p\,\omega_1+q\,\omega_2,
        \qquad
        \omega_{34}=r\,\omega_1+s\,\omega_2.
\]
With the conventions in Section \ref{pre}, the covariant derivatives of the second
fundamental form are defined by
\begin{equation}\label{eq:h-covder}
        \sum_k h^\alpha_{ijk}\omega_k
        =dh^\alpha_{ij}
         +\sum_l h^\alpha_{lj}\omega_{li}
         +\sum_l h^\alpha_{il}\omega_{lj}
         +\sum_\beta h^\beta_{ij}\omega_{\beta\alpha}.
\end{equation}
The Codazzi equations are $h^\alpha_{ijk}=h^\alpha_{ikj}$.
For the third normal direction, \eqref{eq:canonical-shapes} and
\eqref{eq:h-covder} give
\begin{align*}
        \sum_k h^3_{11k}\omega_k&=da,\\
        \sum_k h^3_{12k}\omega_k&=2a\omega_{12}-b\omega_{34},\\
        \sum_k h^3_{22k}\omega_k&=-da.
\end{align*}
Hence the Codazzi identities imply
\[
        a_1=-2aq+bs,
        \qquad
        a_2=2ap-br.
\]
For the fourth normal direction, we similarly obtain
\begin{align*}
        \sum_k h^4_{11k}\omega_k&=-2b\omega_{12}+a\omega_{34},\\
        \sum_k h^4_{12k}\omega_k&=db,\\
        \sum_k h^4_{22k}\omega_k&=2b\omega_{12}-a\omega_{34},
\end{align*}
and therefore
\[
        b_1=-2bq+as,
        \qquad
        b_2=2bp-ar.
\]
Altogether,
\begin{equation}\label{eq:codazzi-a-b}
	\begin{aligned}
		a_1&=-2aq+bs,\qquad a_2=2ap-br,\\
		b_1&=-2bq+as,\qquad b_2=2bp-ar.
	\end{aligned}
\end{equation}
Differentiating \eqref{eq:a-b-constraint} yields
\begin{equation}\label{eq:ai-bi}
        a_i=-\frac{2b}{a}b_i,
        \qquad i=1,2.
\end{equation}
Solving \eqref{eq:codazzi-a-b} with the aid of \eqref{eq:ai-bi}, we find
\[
        q=\frac{3b}{2D}b_1,
        \qquad
        s=\frac{c}{2aD}b_1,
\]
and
\[
        p=-\frac{3b}{2D}b_2,
        \qquad
        r=-\frac{c}{2aD}b_2.
\]
Since $*db=b_1\omega_2-b_2\omega_1$, the formulas in
\eqref{eq:connection-formulas} follow.
\end{proof}

\begin{lemma}\label{lem:grad-b}
On $U$ one has
\begin{equation}\label{eq:grad-b-formula}
        |\nabla b|^2
        =\frac{(c-4b^2)(c-6b^2)
        \bigl(c^2+4cb^2-2c-24b^4\bigr)}
        {6c(c-8b^2)}.
\end{equation}
In particular,
\[
        b^2<\frac c8
        \qquad\text{and}\qquad
        |\nabla b|>0
        \quad\text{on }U.
\]
\end{lemma}

\begin{proof}
Set
\begin{equation}   \label{EG}
	  E(b):=\frac{3b}{c-6b^2},
        \qquad
        G(b):=\frac{c}{a(c-6b^2)}.
\end{equation}
By Lemma~\ref{lem:connections},
\[
        \omega_{12}=E(b)*db,
        \qquad
        \omega_{34}=G(b)*db.
\]
Equations \eqref{eq:star-divergence}, \eqref{eq:gauss-curvature-form},
\eqref{eq:normal-curvature-form}, and \eqref{Ricci} give
\begin{equation}\label{eq:divE-divG}
        E\Delta b+E'|\nabla b|^2=-K,
        \qquad
        G\Delta b+G'|\nabla b|^2=-K^\perp.
\end{equation}
By \eqref{eq:a-b-constraint} and \eqref{eq:K-and-Kperp},
\begin{equation}\label{eq:K-expressions}
        K=1-\frac c2+b^2,
        \qquad
        K^\perp=2ab,
        \qquad
        a^2=\frac{c-4b^2}{2}.
\end{equation}
Eliminating $\Delta b$ from \eqref{eq:divE-divG} yields
\begin{equation}\label{eq:eliminate}
        (GE'-EG')|\nabla b|^2=-GK+EK^\perp=-GK+2Eab.
\end{equation}
A direct differentiation using \eqref{EG} and \eqref{eq:K-expressions} gives
\begin{equation}\label{eq:GE-EG}
        GE'-EG'
        =\frac{3\sqrt2\,c(c-8b^2)}
        {(c-4b^2)^{\frac{3}{2}}(c-6b^2)^2},
\end{equation}
whereas
\begin{equation}\label{eq:rhs-simplified}
        -GK+2Eab
        =\frac{\sqrt2\bigl(c^2+4cb^2-2c-24b^4\bigr)}
        {2\sqrt{c-4b^2}(c-6b^2)}.
\end{equation}
Before dividing by \eqref{eq:GE-EG}, we exclude its only possible zero on
$U$.  If $b^2=\frac{c}{8}$, then the right-hand side of \eqref{eq:eliminate} equals
\[
        \frac{9c-16}{2\sqrt c}>0,
\]
whereas $GE'-EG'=0$, a contradiction.  Thus $c-8b^2\neq0$ on $U$, and
substitution of \eqref{eq:GE-EG} and \eqref{eq:rhs-simplified} into
\eqref{eq:eliminate} proves \eqref{eq:grad-b-formula}.

It remains to determine the sign.  Put $x=b^2$.  By \eqref{eq:b-range},
$0<x<\frac{c}{6}$, so $c-4x>0$ and $c-6x>0$.  Moreover,
\[
        Q(x):=c^2+4cx-2c-24x^2
             =-24x^2+4cx+c(c-2).
\]
The polynomial $Q$ is concave and is positive at both endpoints of
$[0,\frac{c}{6}]$:
\[
        Q(0)=Q(\frac{c}{6})=c(c-2)>0.
\]
Hence $Q(x)>0$ for $0<x<\frac{c}{6}$.  Since the left-hand side of
\eqref{eq:grad-b-formula} is nonnegative and every factor in its numerator is
strictly positive, one must have $c-8b^2>0$.  Formula
\eqref{eq:grad-b-formula} then gives $|\nabla b|>0$.
\end{proof}

\section{Proof of the main theorem}
We first record the analytic continuation in the following lemma.

\begin{lemma}\label{lem:analytic-circular}
Let $M^2\to\Sph^4$ be a minimal immersion.  If
\[
        W:=\{p\in M:\lambda_1(p)=\lambda_2(p)\}
\]
has nonempty interior, then $\lambda_1\equiv\lambda_2$ on $M$.
\end{lemma}

\begin{proof}
	Fix local isothermal coordinates $z=x+iy$ for the induced metric, so that
	\[
	g=e^{2u}(dx^2+dy^2).
	\]
	Regarded as an $\mathbb R^5$-valued map, a minimal immersion into the unit
	sphere satisfies
	\[
	\Delta_gF+2F=0.
	\]
	Since $|F_x|^2=|F_y|^2=e^{2u}$, this equation becomes
	\[
	F_{xx}+F_{yy}
	+\bigl(|F_x|^2+|F_y|^2\bigr)F=0.
	\]
	This is an analytic strongly elliptic system. By Morrey's analytic
	regularity theorem for analytic nonlinear elliptic systems
	\cite{Morrey1958}, the immersion $F$ is real analytic in these isothermal
	(and hence harmonic) coordinates. Consequently, the induced metric is real
	analytic.
	
	After shrinking the coordinate neighborhood, choose real-analytic
	orthonormal tangent and normal frames. In these frames,
	\[
	h^\alpha_{ij}
	=\langle D_{e_i}e_j,e_\alpha\rangle,
	\qquad
	\mathcal A_{\alpha\beta}
	=\langle A^\alpha,A^\beta\rangle
	=\sum_{i,j}h^\alpha_{ij}h^\beta_{ij},
	\]
	are real analytic. Although the ordered eigenvalues of $\mathcal A$ need
	not be real analytic at a multiple eigenvalue, the frame-independent
	discriminant
	\[
	\Phi:=(\lambda_1-\lambda_2)^2
	=(\operatorname{tr}\mathcal A)^2-4\det\mathcal A
	\]
	is real analytic. If $W$ has nonempty interior, then $\Phi$ vanishes on a
	nonempty open set. The identity theorem for real-analytic functions,
	applied on the connected surface $M$, gives $\Phi\equiv0$. Hence
	$\lambda_1\equiv\lambda_2$.
\end{proof}

The next two lemmas eliminate the two possible degenerate configurations.

\begin{lemma}\label{lem:rank-one}
Let $F:M^2\to\Sph^4$ be a minimal immersion such that
$\lambda_2\equiv0$ and $S\equiv c>0$.  Then $c=2$.  Moreover, the image of
$M$ is contained in a totally geodesic $\Sph^3\subset\Sph^4$.
\end{lemma}

\begin{proof}
The fundamental matrix has rank one and the second fundamental form is nowhere
zero.  On a simply connected coordinate neighborhood $V$, choose the normal
frame so that $A^4\equiv0$ and then rotate the tangent frame to obtain
\[
        A^3=\begin{pmatrix}\kappa&0\\0&-\kappa\end{pmatrix},
        \qquad
        A^4=0,
        \qquad
        \kappa=\sqrt{\frac c2}>0.
\]
Write $\omega_{34}=r\omega_1+s\omega_2$.  Applying \eqref{eq:h-covder} to
the fourth normal direction gives
\[
        \sum_k h^4_{ijk}\,\omega_k
        =
        d h^4_{ij}
        +\sum_l h^4_{lj}\,\omega_{li}
        +\sum_l h^4_{il}\,\omega_{lj}
        +\sum_\beta h^\beta_{ij}\,\omega_{\beta 4}
        =h^3_{ij}\omega_{34}.
\]
Taking successively $(i,j)=(1,1),(1,2),(2,2)$ and comparing the
coefficients of $\omega_1$ and $\omega_2$, we obtain
\[
\begin{array}{lll}
 h^4_{111}=\kappa r, & h^4_{112}=\kappa s, &
 h^4_{121}=0,\quad h^4_{122}=0,\\
 h^4_{221}=-\kappa r, & h^4_{222}=-\kappa s. &
\end{array}
\]
Because $h^4_{ij}=h^4_{ji}$ and the Codazzi equations give
$h^4_{ijk}=h^4_{ikj}$, the tensor $h^4_{ijk}$ is symmetric in all three
tangent indices.  Consequently,
\[
        \kappa s=h^4_{112}=h^4_{121}=0,
        \qquad
        -\kappa r=h^4_{221}=h^4_{122}=0.
\]
Since $\kappa>0$, it follows that $r=s=0$, and hence
$\omega_{34}=0$ on $V$.

To spell out the corresponding argument for $A^3$, write
$\omega_{12}=p\omega_1+q\omega_2$.  Since $\kappa$ is constant and
$\omega_{34}=0$, formula \eqref{eq:h-covder} gives
\[
        \sum_k h^3_{11k}\omega_k=0,
        \qquad
        \sum_k h^3_{12k}\omega_k=2\kappa\omega_{12},
        \qquad
        \sum_k h^3_{22k}\omega_k=0.
\]
Thus $h^3_{112}=0$, $h^3_{121}=2\kappa p$,
$h^3_{221}=0$, and $h^3_{122}=2\kappa q$.  The Codazzi identities
$h^3_{112}=h^3_{121}$ and $h^3_{221}=h^3_{122}$ yield
$2\kappa p=2\kappa q=0$.  Hence $p=q=0$, so
$\omega_{12}=0$ on $V$, and \eqref{eq:gauss-curvature-form} yields
$K=0$.  The Gauss equation $K=1-\frac{S}{2}$ now gives $c=2$.

For completeness, $A^4=0$ and $\omega_{34}=0$ also show that, as a vector
field in the ambient space $\R^5$,
\[
        D_Xe_4=-A^4X+\nabla_X^\perp e_4=0
        \qquad (X\in TM).
\]
Thus $e_4$ is a constant vector $v$ on $V$, and
$\langle F,v\rangle=0$ there.  Since $F$ is real analytic and $M$ is
connected, $\langle F,v\rangle$ vanishes identically.  Therefore
$F(M)\subset\Sph^4\cap v^\perp\cong\Sph^3$.
\end{proof}

\begin{lemma}\label{lem:superminimal}
Let $F:M^2\to\Sph^4$ be a minimal immersion such that
$\lambda_1\equiv\lambda_2$ and
$S+\lambda_2\equiv c>0$.  Then $c=2$.  If $M$ is closed, the image of $F$
is the Veronese surface in $\Sph^4$.
\end{lemma}

\begin{proof}
The assertion is local until the final classification, so we may pass to the
oriented double cover and work on a simply connected coordinate
neighborhood.  Put
\[
        \lambda_1=\lambda_2=\lambda.
\]
Since $S=2\lambda$ and $c=S+\lambda_2=3\lambda$, the common eigenvalue
$\lambda=\frac{c}{3}$ is a positive constant.  Choose the orthonormal normal frame
such that \eqref{eq:canonical-shapes} holds, the shape operators take the form
\begin{equation}\label{eq:superminimal-shapes}
        A^3=\begin{pmatrix}a&0\\0&-a\end{pmatrix},
        \qquad
        A^4=\begin{pmatrix}0&a\\a&0\end{pmatrix},
        \qquad
        a=\sqrt{\frac c6}>0.
\end{equation}
Indeed, the two shape operators are orthogonal and have the same squared
norm, equal to $\lambda$.

Write
\[
        \omega_{12}=p\,\omega_1+q\,\omega_2,
        \qquad
        \omega_{34}=r\,\omega_1+s\,\omega_2.
\]
The Codazzi computation used in \eqref{eq:codazzi-a-b}, now with $b=a$ and
$da=0$, gives
\[
        0=-2aq+as,
        \qquad
        0=2ap-ar.
\]
Since $a>0$, we obtain $s=2q$ and $r=2p$, and hence
\begin{equation}\label{eq:normal-tangent-connection}
        \omega_{34}=2\omega_{12}.
\end{equation}
Taking exterior derivatives and using
\eqref{eq:gauss-curvature-form} and \eqref{eq:normal-curvature-form}, we get
\[
        -K^\perp d\mu=d\omega_{34}
        =2d\omega_{12}=-2K\,d\mu,
\]
so that
\begin{equation}\label{eq:Kperp-2K}
        K^\perp=2K.
\end{equation}
On the other hand, the Gauss and Ricci equations applied to
\eqref{eq:superminimal-shapes} yield
\[
        K=1-2a^2,
        \qquad
        K^\perp=2a^2.
\]
Combining these identities with \eqref{eq:Kperp-2K} gives
\[
        2a^2=2(1-2a^2),
        \qquad
        6a^2=2.
\]
Therefore
\[
        c=6a^2=2,
        \qquad
        K=\frac13,
        \qquad
        S=\frac43,
        \qquad
        \lambda_2=\frac23.
\]

Suppose now that $M$ is closed.  The oriented double cover has constant
positive Gauss curvature and hence is a $2$-sphere by the Gauss--Bonnet
theorem.  The lifted immersion is linearly full in $\Sph^4$: otherwise its
image would lie in a totally geodesic $\Sph^3$, and the fundamental matrix
would have rank at most one, contrary to
$\lambda_1=\lambda_2=\frac{2}{3}$.  The Calabi--do Carmo--Wallach classification of
linearly full minimal $2$-spheres of constant curvature
\cite{Calabi1967,DoCarmoWallach1971} therefore identifies the image with the
Veronese surface.
\end{proof}

Now, we are ready to prove the main result.
\begin{proof}[\textbf{Proof of Theorem~\ref{thm:main}}]
Since $S\geq0$ and $\lambda_2\geq0$, one has $c\geq0$.  If $0\leq c\leq2$,
Lu's first-gap theorem applies. Then $M$ must be one of the following:
\[
\begin{array}{c|c|c|c}
\text{$M$}&S&\lambda_2&S+\lambda_2\\ \hline
\text{totally geodesic }\Sph^2&0&0&0\\
\text{Clifford torus in }\Sph^3&2&0&2\\
\text{Veronese surface in }\Sph^4&\frac43&\frac23&2
\end{array}
\]
Thus the conclusion of the theorem holds whenever $c\leq2$.

It remains to rule out $c>2$.  Suppose, to the contrary, that
$S+\lambda_2\equiv c>2$.  If $M$ is nonorientable, pass to its connected
oriented double cover.  The image, the eigenvalues of $\mathcal A$, and the
constant $c$ are unchanged.

Assume first that
\[
        \lambda_2\not\equiv0
        \qquad\text{and}\qquad
        \lambda_1\not\equiv\lambda_2.
\]
Since $M$ is closed, $\lambda_2$ attains a positive maximum at a point $x_0$.
If $x_0\in U$, then locally $\lambda_2=2b^2$, and hence
\[
        0=\nabla\lambda_2(x_0)
         =4b(x_0)\nabla b(x_0).
\]
Because $b(x_0)>0$, this contradicts Lemma~\ref{lem:grad-b}.

It remains to consider a positive maximum point in
\[
        W_+:=\{p\in M:\lambda_1(p)=\lambda_2(p)>0\}.
\]
By Lemma~\ref{lem:analytic-circular}, the set
$W=\{p\in M:\lambda_1(p)=\lambda_2(p)\}$ has empty interior in the present
case.  At $x_0\in W_+$,
\begin{equation}\label{eq:circular-lambda2}
        c=S(x_0)+\lambda_2(x_0)=3\lambda_2(x_0),
        \qquad
        \lambda_2(x_0)=\frac c3.
\end{equation}
Since $\lambda_2(x_0)>0$, the function $\lambda_2$ remains positive in a
neighborhood of $x_0$.  The empty-interior property of $W$ therefore gives a
sequence $x_j\in U$ converging to $x_0$.  By Lemma~\ref{lem:grad-b},
\[
        \lambda_2(x_j)=2b(x_j)^2<\frac c4.
\]
Passing to the limit yields
$\lambda_2(x_0)\leq \frac{c}{4}$, contradicting
\eqref{eq:circular-lambda2}.

Consequently, one of the two global alternatives must hold:
\[
        \lambda_2\equiv0
        \qquad\text{or}\qquad
        \lambda_1\equiv\lambda_2.
\]
The first alternative contradicts Lemma~\ref{lem:rank-one}, which forces
$c=S=2$.  The second contradicts Lemma~\ref{lem:superminimal}, which also
forces $c=2$.  Hence no constant value $c>2$ is possible.  Combining this
with the first-gap classification proves the theorem.
\end{proof}

   {\noindent\bf Acknowledgements }
The authors would like to thank Dr. Zhao Yuhang for pointing out that our proof yields a classification result, namely Theorem \ref{thm:main}.

\end{document}